\documentclass[11pt]{amsart}

\usepackage{color}

\usepackage{dsfont}
\usepackage{amsmath}
\usepackage{amssymb}
\usepackage{graphicx}

\usepackage{amssymb}

\usepackage{amsfonts}

\usepackage{soul}

\usepackage{mathpazo}
\usepackage{color}
\usepackage{yfonts}

\usepackage{paralist}

\usepackage{stmaryrd}

\usepackage{amsxtra}

\def\cU{\mathcal U}

\def\eps{\varepsilon}

\newcommand{\cR}{\mathcal R}

\newcommand{\norm}[1]{\| #1\|}

\def\a{          \alpha}

\def\cW{          \mathcal W}
\def\cWu{          \mathcal W^{u}}
\def\cWs{          \mathcal W^{s}}

\let\cal\mathcal

\def \R{{\mathbb R}}

\def \Z{{\mathbb Z}}
\def \N{{\mathbb N}}

\newcommand{\T}{{\mathbb T}}
\newcommand{\prf}{{\begin{proof}}}
\newcommand{\epf}{{\end{proof}}}

\newcommand{\Q}{{\mathbb Q}}

\newtheorem{prop}{\sc Proposition}

\newtheorem{lemma}{\sc lemma}
\newtheorem{cor}{\sc corollary}

\theoremstyle{definition}

\newtheorem{Main}{Theorem}

\def\bee{\begin{equation}}
\def\eee{\end{equation}}

\newtheorem{defi}{\sc Definition}
\newtheorem{nota}{\sc Notation}

\theoremstyle{rema}

\newtheorem{rema}{\sc Remark}

\newcommand{\pdvr}[2]
{\dfrac{\partial^{#2} #1}{\partial \theta^{#2_1} \partial r^{#2_2}}}
\newcommand{\pdvrs}[2]
{\partial^{#2} #1 /\partial \theta^{#2_1} \partial r^{#2_2}}
\newtheorem{example}{Example}

\numberwithin{equation}{section}

\author{Bassam Fayad  and  Zhiyuan Zhang}

\begin{document}

%
%
%
%

\title[Effective Katok's horseshoe theorem ]{An effective version of Katok's horseshoe theorem for conservative $C^2$ surface diffeomorphisms}
\date{\today}
\maketitle

\begin{abstract}
For area preserving $C^2$ surface diffeomorphisms,  we give an explicit finite information condition, on the exponential  growth of the number of Bowen's $(n,\delta)-$balls  needed to cover a positive proportion of the space, that is sufficient to guarantee positive topological entropy. This can be seen as an effective version of Katok's horseshoe theorem in the conservative setting. We also show that the analogous result is false in dimension larger than $3$.
\end{abstract}

\section{Introduction}

 Let $X$ be a compact smooth surface with a Riemannian metric.
Denote by ${\rm Diff}^{r}_{\text{vol}}(X)$ the group of $C^r$ diffeomorphisms which preserve the volume form $m$ induced by the Riemannian metric.  Without loss of generality, we assume that $m(X) = 1$.

A well-known result of Katok, based on Pesin theory, says that  if $f \in {\rm Diff}^{1+\epsilon}(X)$ has 
 non-zero Lyapunov exponent for some $f-$invariant non atomic measure, then the topological entropy of $f$ is positive and that $f$ actually has invariant horseshoes that carry most of the topological entropy  (see for example \cite{katok_ihes}, or \cite{kh}). In particular, this is the case for any $f \in {\rm Diff}^{1+\epsilon}_{\text{vol}}(X)$ having positive Lyapunov exponents on a positive measure set, or in other words, when $f$ has positive metric entropy by Pesin's formula. 

Besides the positivity of Lyapunov exponents, another manifestation of positive metric entropy is the exponential rate of growth of the Bowen $(n,\delta)-$balls (see Definition \ref{def:bowen}) that are needed to cover a definite proportion of $X$ (see for example  \cite{kh}).

\begin{defi} \label{def:bowen}Given a continuous map $f : X \to X$.
For any $\delta > 0$, integer $n\geq 1$, any $x \in X$, we define Bowen's $(n,\delta)-$ball centered at $x$ by
\begin{align*}
B_f(x, n, \delta) = \{y | d(f^{i}(x), f^{i}(y)) < \delta, \forall 0 \leq i \leq n-1 \}
\end{align*}
Given an $f-$invariant measure $\mu$.
For any $\varepsilon \in ( 0, 1 )$, let
$N_{f}(n, \delta, \varepsilon) = \inf_{\cU} |\cU|$, where $\cU$ is taken over all the subsets of $\{B_f(x, n, \delta)\}_{x \in X}$ such that the union of $(n,\delta)-$balls in $\cU$ has $\mu-$measure not less than $1-\varepsilon$. For a finite set $I$, we use $|I|$ to denote the cardinality of $I$.
\end{defi}

By the sub-additive growth of the number of Bowen balls and Katok's horseshoe theorem, the following statement is true by compacity: 

\textit{Fact: If the $C^2$ norm of $f$ is bounded by $D>0$, and if $h,\delta,\varepsilon>0$ are fixed, then there exists $n_0 = n_0(D,h,\delta,\eps)>0$ such that if   
$N_{f}(n, \delta, \varepsilon) \geq e^{nh}$ for some integer $n > n_0$, then $f$ has positive topological entropy. 
}

{\it Sketch of proof.} Assume by contradiction that there exists $h,\delta,\varepsilon>0$ and a sequence $f_{n}$ with a uniform bound on its $C^2$ norm  for which $N_{f_n}(n, \delta, \varepsilon) 
\geq e^{nh}$ and $h_{\rm top}(f_{n})=0$. By compacity we can, up to passing to a subsequence, assume that $f_n$ has a limit $f$ that is $C^{1+{\rm Lip}}$. Since for any $g$, the minimal number   $N_{g}(n, \delta)$ needed to cover all of $X$ is essentially sub-additive in $n$, we have that for a fixed $k \in \N$,  and for any $n$ sufficiently large $N_{f_{n}}(k, \delta) \geq e^{kh/2}$. Therefore   $N_{f}(k, \delta) \geq e^{kh/2}$   for any $k \in \N$
and hence $f$ has positive topological entropy. By Katok's horseshoe theorem, this contradicts the assumption $h_{\rm top}(f_n)=0$ for all $n$.  \hfill $\hfill \Box$



In this paper, we will give a direct proof of the above fact that also provides an explicit upper bound for  $n_0(D,h,\delta,\eps)$. Our bound will essentially be a tower-exponential of height $K \sim \log(\frac{\log A}{h})$ where $A=\|f\|_{C^1}$. 
The norm of the second derivative of $f$ enters into the argument of the tower-exponential bound. We will not use in our proof {\it any}  ergodic theory.

Our main tool is a finite information closing lemma for a map $g\in {\rm Diff}^{2}_{\text{vol}}(X) $ that generalizes the one obtained in \cite[Theorem 4]{AFLXZ}. 
Theorem 4 in \cite{AFLXZ} asserts that 
 if $x$ is such that $\norm{Dg^q(x)}$ is comparable to $\norm{Dg}^{\theta q}$ where  $ \theta$ is close to $1$ and  $q$ is sufficiently large compared to powers of the $C^2$ norm of $g$, then there exists a hyperbolic periodic point that shadows a piece of a length $q$ orbit of $x$. A similar effective closing lemma was previously obtained by Climenhaga and Pesin in \cite{CP} for  $C^{1+\epsilon}-$diffeomorphisms in any dimension, assuming however the existence of a splitting of the tangent spaces along a long orbit with some additional estimates of effective hyerbolicity.  For an interesting application of the latter effective approach, we refer the reader to \cite{CDP}. 
 
  In this note we will need a generalized version of the effective closing lemma in \cite{AFLXZ} that gives a shadowing of $x$ by a hyperbolic periodic orbit, even when $\norm{Dg^q(x)}$ is much smaller than $\norm{Dg}^{\theta q}$,  provided that $ \norm{Dg^q(x)} \geq \norm{Dg(g^i(x))}^{\theta q}$, for most of the $i \in [0,q]$. 
An inductive use of this closing lemma allows one to obtain, under the growth condition of the $(n,\delta)$-balls, sufficiently many hyperbolic periodic points with a good control on their local stable and unstable manifolds to insure the existence of a horseshoe. 
Note that, in order exploit the growth condition of the Bowen balls, we need sufficiently precise informations from the shadowing property, which are  not covered by the direct bootstrapping of Theorem 4 in \cite{AFLXZ}.

With the same approach, we are also able to conclude positive topological entropy from derivative growth at an explicit time scale along a single, yet not too concentrated, orbit.

\subsection{Statements of the main results}

Throughout this note,  $X$ is a compact surface with a  volume form $m$. Without loss of generality, we assume that $m(X) = 1$.
We will denote by $f : X \to X$ a $C^2$ diffeomorphism that preserves $m$ such that for constants $A,D > 0$, 
\begin{align} \label{norms} \tag{$*$}
\left\{ \begin{array}{ll}  \norm{Df}&\leq A, \\   \norm{D^2f}&\leq D.  \end{array}
\right. \end{align} 
Here $\norm{Df}$, $\norm{D^2f}$  denote respectively the supremum of the first and second derivatives of $f$.

All the constants that appear in the text will implicitly depend on the surface $X$.


To simplify notations, we define the following.

\begin{defi}
For $R_0,R_1 > 0$, $K \in \Z_{+}$, we define function $Tower: \R_{+}^2 \times \Z_{+} \to \R$ by the following recurrence relation,
\begin{eqnarray}
Tower(R_0, R_1, K) = \begin{cases} 
                 R_0, & K=1, \\
                 R_1^{Tower(R_0, R_1, K-1 )}, & K \geq 2.
               \end{cases}
\end{eqnarray}
\end{defi}

Our main result is the following.

\begin{Main}\label{mainA} There exists a constant $C_0 = C_0(X) >0$ such that the following is true. For any  $A, D > 1$, $h \in ( 0, \log A]$, $\varepsilon \in (0,1)$, $\delta > 0$, denote by
\begin{eqnarray}\label{defQ} 
P_0  &=&  \max( \varepsilon^{-1}e^{C_0(\log(\frac{\log A}{h}))^2 + C_0}, C_0h^{-1}\log \delta^{-1} ), \\
P_1 &=&  e^{C_0 h^{-1} \log D \log A} \label{defQ1}.
\end{eqnarray}
If $f : X \to X$ is a $C^2$ diffeomorphism preserving $m$ that satisfies \eqref{norms}, and $N_{f}(n, \delta, \varepsilon) > e^{nh}$ for some $n \geq Tower(P_0, P_1,K_0)$, where $K_0 =  \lceil C_0 \log(\frac{\log A}{h}) + C_0 \rceil $, then 
$f$ has positive topological entropy.
 \end{Main}

Theorem \ref{mainA} gives positive topological entropy from complexity growth at an explicit large time scale.
Some adaptation of the proof also allows us to conclude positive topological entropy from derivative growth at an explicit time scale along a single, yet not too concentrated, orbit. To precisely formulate such a result, we introduce the following notation.

\begin{defi}
Given a continuous map $f : X \to X$, for any subset $I \subset \Z$, any $x \in X$, we set $Orb(f,x,I) = \{ f^{i}(x) | i \in I\}$. 

For constants $c, \delta > 0, \varepsilon \in (0,1)$, we say that $x$  is \textit{ $(n, c,\delta,\varepsilon)$-sparse }  if for any subset $I \subset \{0,\cdots, n-1 \}$ satisfying $|I| > c n$ we have    $m(B( Orb(f,x, I), \delta)) > \varepsilon$. 
\end{defi}

\begin{Main}\label{mainB}
There exists a constant $C_0 = C_0(X) > 0$ such that the following is true.
For any $A, D > 1$, $h \in ( 0, \log A]$, $\varepsilon \in (0,1)$, let
\begin{eqnarray*}
P_0  = \varepsilon^{-1}e^{C_0(\log(\frac{\log A}{h}))^2 + C_0}, \quad P_1 =  e^{C_0 h^{-1} \log D \log A}.
\end{eqnarray*}
 If $f : X \to X$ is a  $C^2$ diffeomorphism preserving $m$ that satisfies \eqref{norms}, and there exists $x \in X$ such that for some $n \geq  Tower(P_0,P_1,K_0)$, where $ K_0 =  \lceil C_0 \log(\frac{\log A}{h}) + C_0 \rceil$, we have
\begin{itemize} 
\item  $\norm{Df^{n}(x)} > e^{nh}$,
\item $x$ is $(n, Tower(P_0,P_1,K_0-1)^{-1}, D^{-Tower(P_0, P_1, K_0-1)},\varepsilon)$-sparse,
\end{itemize} 
then $f$ has positive topological entropy.
\end{Main}

Observe that a non-concentration condition, such as the second condition of Theorem \ref{mainB}, is necessary to conclude positive entropy, for otherwise $x$ could just belong to a  hyperbolic periodic orbit with a small period.

We remark that Theorem \ref{mainA} does not hold in general in dimension at least $4$ as the following example shows.

\begin{example} Denote by  $\{ g_t \}_{t \in \R}$  a geodesic flow on $Y := S_1M$, the unit tangent space of a hyperbolic surface $M$, preserving the Liouville measure $\mu$. We set $h_0 := h_{\mu}(g_1) > 0$. Let $\T = \R / \Z$ be the circle and let $\varphi \in C^\infty(\T)$ be a function such that  $\int_{\T} \varphi d\theta = 0$ and  $\varphi |_{[0, \frac{1}{2}]} \equiv 1$. For  any $\alpha \in \R$, denote by $R_{\alpha} : \T \to \T$ the rotation $\theta \mapsto \theta + \alpha[1]$, and consider the $C^2$ map $f_{\alpha} : \T \times Y \to \T \times Y$ defined as follows,
\begin{eqnarray*}
f_{\alpha}(\theta,x) = (\theta+\alpha, g_{\varphi(\theta)}(x)),\quad \forall (\theta,x) \in \T \times Y.
\end{eqnarray*}
\end{example}

Observe that for any $\alpha \in \R$, $f_{\alpha}$ preserves the smooth measure $\nu := Leb_{\T} \times \mu$. It is clear that  $\sup_{\alpha \in \T} \norm{f_{\alpha}}_{C^2} < \infty$.
Moreover, we have the following that shows that Theorem \ref{mainA} does not hold in general in dimension at least $4$.

\begin{prop} We have that \label{prop counter}
\begin{itemize} 
\item[(1)] For any $\alpha\in \R - \Q$, the topological entropy $h_{top}(f_{\alpha})=0$.

\item[(2)] There exists $\delta > 0$ such that for any $\varepsilon \in (0,1)$, any integer $n_0 > 0$, there exists $n > n_0$, $\bar{\alpha} \in \T$, such that for any $\a \in [0, \bar{\a}]$ it holds that $N_{f_{\alpha}}(n, \delta, \varepsilon) > e^{\frac{nh_0}{2}}$. 
\end{itemize} 
\end{prop}

\begin{proof}
Abramov Rohlin formula for the entropy of a skew product yields (1) \cite{AR}. To see (1) directly, let $(q_n)_{n \in \N}$ be the sequence of denominators of the best rational approximations of $\alpha$. Then by Denjoy-Koksma theorem, the partial sums $S_{q_n}\varphi$ defined as $S_{q_n}\varphi(\theta) := \sum_{i=0}^{q_n-1} \varphi(\theta + i \alpha)$,$\forall \theta \in \T$, converge uniformly in the $C^\infty$ topology to $0$, as $n$ tends to infinity. By direct computations, we see that
\begin{eqnarray*}
f_{\alpha}^{q_n}(\theta,x) = (\theta + q_n\alpha, g_{S_{q_n}(\theta)}(x)),\quad \forall (\theta,x) \in \T \times Y.
\end{eqnarray*}
This implies that $f_{\alpha}^{q_n}$ converge to $Id$ in the $C^{\infty}$ topology, as $n$ tends to infinty. By Ruelle's entropy inequality, such convergence can  happen only if $h_{top}(f_{\alpha}) = 0$. 

To see (2), we notice that by $h_{\mu}(g_1) = h_0 > 0$, there exists $\delta > 0$, such that for any $\varepsilon \in (0,1)$, any $n_0 > 0$, there exists $n > n_0$ such that $N_{g_1}(n, \delta, \varepsilon) > e^{ \frac{nh_0}{2}}$. Then by choosing $\alpha$ to be sufficiently close to $0$, so that $i\alpha \in [0, \frac{1}{2}]$ for all $0 \leq i \leq n$, we have $f_{\alpha}^{i}(\theta,x) = (\theta+i\alpha, g_{i}(x))$ for any $(\theta,x) \in \T \times Y$, any $0 \leq i \leq n$. Then it is direct to see that $N_{f_{\alpha}}(n, \delta, \varepsilon) \geq N_{g_1}(n, \delta, \varepsilon) > e^{\frac{nh_0}{2}}$. This concludes the proof.
\end{proof}

\begin{nota}

For any $n \geq 1$, any $x \in X$, we will denote by $\mu_{x,n} = \frac{1}{n}\sum_{m=0}^{n-1}\delta_{f^{m}(x)}$.
For any $x \in X$, any linear subspace $E \subset T_{x}X$, any $r > 0$, we denote by  $B_{E}(r) = \{v \in E | \norm{v} < r \}$. For any subset $A \subset X$, any $r > 0$, we denote by $B(A, r) = \{x | d(x,A) < r\}$. For any measurable subset $K \subset X$, we use $|K|$ or $m(K)$ to denote the measure of $K$.

We will use $c, c_1, \cdots$ to denote generic positive constants which are allowed to vary from line to line, and may or may not depend on $X$, but independent of everything else. Under our notations, expressions like $cA \leq B \leq cA$ are legitimate. For two  variables $A,B > 0$, we denote $A \gg B$ ( resp. $A \ll B$ ) if we have $A \geq c B$ ( resp. $cA \leq B$ ) for some constant  $c$ as above.  

\end{nota}

\section{From hyperbolic points to positive entropy}

\begin{defi} \label{def hyp pts}
Let $g : X \to X$ be a $C^1$ diffeomorphism. For $ \alpha \in (0,\pi), r \in (0,1)$, a hyperbolic periodic point of $g$, denoted by $y \in X$, is said to be $( \alpha, r)-$hyperbolic if the following is true. Let $E^{s}(y), E^{u}(y)$ be respectively the stable and unstable direction at $y$. Then
\begin{enumerate}
\item The angle between $E^s(y)$ and $E^u(y)$ is at least $\alpha$,
\item The local stable (resp. local unstable ) manifold of $g$ at $x$ contains $\exp_{y}(graph(\gamma_s))$ (resp. $\exp_{y}(graph(\gamma_u))$), where $\gamma_s : B_{E^{s}(y)}( r) \to E^{u}(y)$  (resp. $\gamma_u : B_{E^{u}(y)}( r ) \to E^{s}(y)$) is a Lipschitz function such that $\gamma_s(0)=0$ and $Lip(\gamma_s) < \frac{1}{100}$ (resp. $\gamma_u(0) = 0$ and $Lip(\gamma_u) < \frac{1}{100}$).
\end{enumerate}
Moreover, we denote $\exp_y(graph(\gamma_s))$ ( resp. $\exp_y(graph(\gamma_u))$ )  by $\cWs_{r}(y)$ (resp. $\cWu_r(y)$).

For any $\alpha \in (0,\pi), r > 0$, the set of all $(\alpha,r)-$hyperbolic points of $g$ is denoted by $\mathcal{H}(g, \alpha,r)$. 
To simplify notations, for any $\lambda \in (0,1)$, a $(\lambda^{2}, \lambda^{3})-$hyperbolic point of $g$ is said to be $\lambda-$hyperbolic.
The set of all $\lambda-$hyperbolic points of $g$ is denoted by $\mathcal{H}(g, \lambda)$. 
\end{defi}

\begin{defi}[Heteroclinic intersection]
For any $C^1$ diffeomorphism $g : X \to X$, for any two distinct hyperbolic periodic points of $g$ denoted by $p,q$, we say that $p,q$ has a heteroclinic intersection, if the stable submanifold of $p$ intersects transversely the unstable manifold of $q$, and the unstable submanifold of $p$ intersects transversely the stable manifold of $q$.
\end{defi}

The following simple lemma shows that for any given $\alpha,r$, there cannot be too many $(\alpha, r)-$points unless there is a heteroclinic intersection.

\begin{prop}\label{lem alt} 
There exist $C_1,C_2 > 1$ depending only on $X$ such that, for any $\alpha  \in (0,\pi)$, any $0 < r < C_1^{-1} $, if a $C^1$ diffeomorphism $g : X \to X$ satisfy $|\cal{H}(g, \alpha, r)| > C_2 r^{-2} \alpha^{-4}$, then there exists a heteroclinic intersection for $g$. In particular, $g$ has positive topological entropy.   In particular, if $\lambda \ll 1$ and $\mathcal{H}(g, \lambda) \gg \lambda^{-14}$, then there exists a heteroclinic intersection for $g$.
\end{prop}

\begin{proof}

In order to be able to measure the angles between vectors in nearby tangent spaces, we cover the surface $X$ by finitely many $C^{\infty}$ local charts  $\{ \psi : [-1,2]^2 \to X \}_{\psi \in \cal{B}}$ indexed by $\cal{B}$.  
For any three distinct points $x,y,z \in \R^2$, let $\angle(x,y,z)$ denote $\angle(x-y, z-y)$.
For any $\beta > 0$, any $v \in \R^2 \setminus \{0\}$, let $C(v, \beta) :=\{ u | \angle(u,  v) < \beta \} \bigcup \{0\}$.

 We will choose $\{ \psi : [-1,2]^2 \to X \}_{\psi \in \cal{B}}$ and  a constant $c_0 > 0$, depending only on X, such that for any $x \in X$, any $\psi \in \cal{B}$ such that $x \in \psi([-0,1]^2)$, for any $v_1, v_2 \in T_xX \setminus \{0\}$, set $\hat{x} := \psi^{-1}(x), \hat{v}_1 := D\psi^{-1}(x,v_1), \hat{v}_2 := D\psi^{-1}(x,v_2)$, then  :
\begin{enumerate}
\item
$2^{-1} \angle(v_1, v_2) \leq \angle( \hat{v}_1, \hat{v}_2) \leq 2 \angle(v_1,v_2)$,
\item If $\norm{v_1}, \norm{v_2}  < 2c_0^{-1}$, then $\psi^{-1} exp_x(v_i)$ is defined and
\begin{eqnarray*}
2^{-1} \angle(v_1,v_2) \leq \angle( \psi^{-1}\exp_x(v_1) , \hat{x}, \psi^{-1}\exp_x(v_2) ) \leq 2 \angle(v_1,v_2).
\end{eqnarray*}
\end{enumerate}
 We fix an arbitrary smooth measure $\hat{m}$ on compact manifold
\begin{eqnarray*}
\widehat{X} = \{(x,v_1,v_2) | x \in X, v_1,v_2 \in T_xX, \norm{v_1}= \norm{v_2}=c_0^{-1}\}.
\end{eqnarray*}
Let $c_1 > 0$ be a large constant to determined later, and for any $(x,v_1,v_2) \in \widehat{X}$, any $\psi \in \cal B$ so that $x \in \psi((0,1)^2)$ and set $$Q_{\psi}(x,v_1,v_2) = \{(y, u_1,u_2) \in \widehat{X} | |\hat{x} - \hat{y}| <\frac{r \alpha}{c_1} , \angle(\hat{v}_1, \hat{u}_1), \angle(\hat{v}_2 , \hat{u}_2) < \frac{\alpha}{40}  \}.$$
Then there exists $c_2  > 0$ depending only on $X, c_1$, such that for all $(x,v_1,v_2) \in \widehat{X}$, any $\psi \in \cal B$ so that $x \in \psi((0,1)^2)$, we have
$$\hat{m}(\psi(Q_{\psi}(x,v_1,v_2))) > c_2^{-1}r^2\alpha^4.$$

By pigeonhole principle,  there exists a constant $c_3 > 0$ depending only on $X, c_2$, such that whenever $|\cal H(g, \alpha,r)| > c_3 r^{-2}\alpha^{-4}$, there exists a chart $\psi \in \cal{B}$, $(y_i, v^{s}_i, v^{u}_i) \in \widehat{X}$, $i=1,2$ so that
\begin{enumerate}
\item $y_1,y_2 \in \cal{H}(g, \alpha, r) \bigcap \psi((0,1)^2)$ are two distinct points;
\item for $i=1,2$, $\angle(v^{s}_i, v^{u}_i) \leq \frac{\pi}{2}$,  and $v_i^s$ (resp. $v_i^u$) lies in the stable (resp. unstable) direction of $y_i$;
\item $Q_{\psi}(y_1,v_1^s,v_1^u) \bigcap Q_{\psi}(y_2,v_2^s,v_2^u) \neq \emptyset.$
\end{enumerate}
 This implies that
$|\hat{y}_1 - \hat{y}_2| < \frac{2r \alpha}{c_1}$, $\angle(\hat{v}_1^{s}, \hat{v}_2^{s}) < \frac{\alpha}{20}$ and $\angle( \hat{v}_1^{u}, \hat{v}_2^{u}) < \frac{\alpha}{20}$.

For $i=1,2$, let us denote $\alpha_i = \angle(v_i^u,  v_i^s)$. By the definition of $\cal H(g, \alpha, r)$ we have $\alpha_1, \alpha_2 \geq \alpha$.
Then  $\angle(\hat{v}_i^u,  \hat{v}_i^s)  \geq 2^{-1}\alpha_i$ for $i=1,2$. Moreover for $r \ll 1$, we have $\psi^{-1}(\cW^{u}_r(y_i) ) \subset \hat{y}_i + C(\hat{v}_i^{u}, \frac{1}{20} \alpha_i)$ since there exists $\gamma_u : B_{E^u(y_i)}(r) \to E^s(y_i)$ with $Lip(\gamma_u) < \frac{1}{100}$, such that $\cWu_r(y_i) = exp_{y_i} graph(\gamma_u)$ and $graph(\gamma_u) \subset C(v_i^u, \frac{1}{40} \alpha_i)$. Similarly, we have $\psi^{-1}(\cW^{s}_r(y_i)) \subset \hat{y}_i +  C(\hat{v}_i^s, \frac{1}{20} \alpha_i)$.

By straightforward calculations, when $c_1$ is chosen to be sufficiently large, $y_1$,$y_2$ above have a heteroclinic intersection. Thus for any $r \ll 1$, any $C^1$ diffeomorphism $g : X \to X$ so that  $|\cal{H}(g, \alpha, r)| \gg r^{-2} \alpha^{-4}$, there exists a heteroclinic intersection for $g$.
 It is a standard fact that for $C^1$ surface diffeomorphism, the existence of a heteroclinic intersection implies positive topological entropy.  This concludes the proof.
\end{proof}

\section{A closing lemma}

\begin{defi}
For any $\eta > 0$, any integer $l > 0$, any $C^0$ map $g : X \to X$, any subset $Y \subset X$, a point $x \in X$ is said to be $(\eta, l,g)-$recurrent for $Y$ if we have
\begin{align*}
\frac{1}{l}|\{ 0 \leq j  \leq l-1| g^{j}(x) \in Y \}| > \eta.
\end{align*}
For any subset $Y \subset X$, we denote by
\begin{eqnarray*}
\mathcal{R}(Y, \eta, l, g) := \{ (\eta, l, g)-\mbox{recurrent points for $Y$ }\}.
\end{eqnarray*}
\end{defi}

For any $\lambda, \xi > 0$, we set 
\begin{eqnarray} 
\label{defG} \cal{G}(\lambda, \xi, g):= \bigcup_{y \in \mathcal{H}(g, \lambda)} B( \cWu_{\lambda^{3}}(y), \xi).
\end{eqnarray}

By our definition, we clearly have $\cal{G}(\lambda, \xi, g) = \cal{G}(\lambda, \xi, g^k)$ for any $ k \geq 1$, since $\mathcal{H}(g, \lambda) = \mathcal{H}(g^k, \lambda)$ for any $k \geq 1$.

Theorem 4 in \cite{AFLXZ} can be strengthened to prove the following proposition. 

\begin{prop} \label{lem hyp pt}
There exist $C = C(X) > 1$, and an absolute constant $\theta_0  \in (\frac{1}{2},1)$ such that the following is true. For each $\Delta \geq 1$, we set
\begin{eqnarray} \label{def eta L}
\eta = \eta(  \Delta)  :=  C^{-1}\Delta^{-2}\in (0,1)
\end{eqnarray}
Let $g : X \to X$ be a $C^2$ diffeomorphism preserving $m$. If for $A_1 \geq C$, $D_1 \geq A_1$, an integer $q \geq D_1^{C \Delta }$ and  $x \in X$, we have the following :
\begin{enumerate}
\item $\norm{Dg} \leq A_1^{\Delta}$ ,
\item $\norm{D^2g} \leq D_1$ ,
\item $x \notin \cal R(\{y | \norm{Dg(y)}>A_1^{\theta_{0}^{-1}} \}, \eta, q, g)$,
\item $\norm{Dg^{q}(x)} > A_1^{q}$,
\end{enumerate}
then 
\begin{eqnarray*}
x \in \cal F(A_1,D_1,  \Delta,  q , g) := \bigcup_{1 \leq j \leq q} g^{-j }( \cal{G}( D_1^{-C  \Delta}, A_1^{-\frac{ q }{2D_1^{C  \Delta}}}, g) ) 
\end{eqnarray*}

\end{prop}

The proof of Proposition \ref{lem hyp pt} follows closely that of Theorem 4 in \cite{AFLXZ}. In our case we need to get more precise informations on the regularity of local invariant manifolds, as well as the location of the hyperbolic point. We defer its proof to Appendix A relying on many estimates from \cite{AFLXZ}.

\section{Estimates along a tower exponential sequence}\label{Estimates along a tower exponential sequence}

Without loss of generality, we will always assume  that $D,A$ in Theorem \ref{mainA}, \ref{mainB} satisfy
\begin{eqnarray} \label{cond 1}
D > A \gg 1.
\end{eqnarray}
Then we can assume that for any $C^2$ map $g : X \to X$ such that $\norm{Dg}, \norm{D^2g} \leq D$, we have 
\begin{eqnarray*}  
\norm{D^2g^k} < D^{k}, \quad \forall k \geq 1.
\end{eqnarray*}

Let $C, \theta_0$ be defined in Proposition \ref{lem hyp pt}. For $D,A,h$ given in Theorem \ref{mainA} or \ref{mainB},  set $C'$ to be a large positive constant depending only on $X$ to be determined later. We set
\begin{eqnarray}\label{def K Delta}
 \quad  \Delta = \frac{16 \log A}{h}, \quad K = \lceil \frac{\log(\frac{\Delta}{4})}{-\log \theta_0} \rceil \geq 2, \quad \eta = \eta(\Delta) \mbox{ ( see \eqref{def eta L}  )}.
\end{eqnarray}
Define
\begin{eqnarray} \label{HHH}
H = H (X, A, h)   := C' \Delta.
\end{eqnarray}

Given an integer $n \geq 1$, $\varepsilon \in ( 0, 1)$, we inductively define the following.
\begin{eqnarray}
q_0 &=&   \lceil \varepsilon^{-1}  e^{C'(\log \Delta)^{2} } \label{initial q} \rceil   ,  \\
l_k&=&\left\{ \begin{array}{cl}
                \lceil D^{q_kH}  \rceil  \ , &0 \leq k \leq K-1  \\[1mm]
                \lceil \frac{n }{  q_{K}} \rceil , & k = K \\[1mm]
               \end{array} \right.  , \quad 
  q_{k+1} = q_{k}l_k \label{def l q}.
\end{eqnarray}
For $0 \leq k \leq K$, we set
\begin{eqnarray} 
\lambda_k = D^{-C  \Delta q_k}, \quad \xi_k = A^{-\frac{q_{k+1}\theta_0^{k+1}}{2D^{C  \Delta q_k}}},
\end{eqnarray}
and set
\begin{eqnarray*}
Q_0  = \varepsilon^{-1}e^{C'(\log \Delta)^2 }, \quad  Q_1 =  e^{20C' h^{-1} \log D \log A} .
\end{eqnarray*}

We have the following simple lemma.
\begin{lemma} \label{lemmacollectionsofproperties}
\begin{enumerate}
\item $e^{\frac{h}{16}} < A^{\theta_0^{K+1}} \leq e^{\frac{h}{4}}$,
\item For any $C' \gg 1$, for all $0 \leq k \leq K-1$, we have $D^{q_k H} \leq l_{k} \leq Tower(Q_0, Q_1, k+2)$. If $n > Tower( Q_0, Q_1, K+3)$, then $l_{K}  \geq D^{q_K H}$,
\item  
For any $C' \gg 1$, set $\delta_0 = D^{-Tower( Q_0,  Q_1,K+1)}$, we have
\begin{eqnarray*}
&& \xi_i \leq \xi_0, \quad  \delta_0 < \min( \lambda^3_K, \xi_0), \\
 && C' \lambda_i^{-11} \max( \delta_0 , \xi_i) < \varepsilon, \quad \forall 0 \leq i \leq K.
\end{eqnarray*}

\end{enumerate}
\end{lemma}

We define for $0 \leq k \leq K$,
\begin{eqnarray}
\cal{G}_k &:=&  \cal{G}(\lambda_k, \xi_k, f), \label{def G k} \\
 \label{def F k}
\cal{F}_k &:=& \cal{F}(A^{q_k\theta_0^{k+1}}, D^{q_k},  \Delta ,  l_k, f^{q_k}).
\end{eqnarray}

The following is a corollary of Proposition \ref{lem hyp pt}.

\begin{cor}\label{corofproplemhyppt}
If $n > Tower(Q_0, Q_1, K+3)$, then for any $0 \leq k \leq K$ we have
\begin{eqnarray*}
x \notin \cal R( \{ y | \norm{Df^{q_{k}}(y)} > A^{q_k \theta_0^{k}} \}, \eta(\Delta), l_k, f^{q_k}) \bigcup \cal F_k \implies \norm{Df^{q_{k+1}}(x)} \leq A^{q_{k+1}\theta_0^{k+1}}.
\end{eqnarray*}
\end{cor}
\begin{proof}
By Lemma \ref{lemmacollectionsofproperties}(2), if $n > Tower(Q_0, Q_1, K+3)$ then for any $0 \leq k \leq K$, we have $l_K \geq D^{q_k H}$.
By our choice of $A, D$, we have 
\begin{eqnarray*}
\norm{Df^{q_k}} \leq A^{q_k}, \quad \norm{D^2f^{q_{k}}} < D^{q_k}, \quad \forall 0 \leq k \leq K.
\end{eqnarray*}
We take any $0 \leq k \leq K$, and an arbitrary point $x \in X$ such that $\norm{Df^{q_{k+1}}(x)} > A^{q_{k+1}\theta_0^{k+1}}$. It suffices to show that $x \in \cal R( \{ y | \norm{Df^{q_{k}}(y)} > A^{q_k \theta_0^{k}} \}, \eta(\Delta), l_k, f^{q_k}) \bigcup \cal F_k $.  By Lemma \ref{lemmacollectionsofproperties}(1) we have $\norm{Df^{q_{k}}(x)} \leq A^{q_{k}} \leq (A^{q_k \theta_0^{k+1}})^{\frac{16 \log A}{h}}$. If $x \in \cal R( \{ y | \norm{Df^{q_{k}}(y)} > A^{q_k \theta_0^{k}} \}, \eta(\Delta), l_k, f^{q_k})$, we are done. Otherwise, we can verify conditions (1)-(4) in Proposition \ref{lem hyp pt} for $(f^{q_k}, A^{q_k \theta_0^{k+1}}, D^{q_k},  \frac{16 \log A}{h}, l_{k})$ in place of $(g, A_1, D_1, \Delta, q)$. We can apply Proposition \ref{lem hyp pt} for map $g = f^{q_k}$ to show that $x \in \cal F_k$. This completes the proof.
\end{proof}

The following is a straightforward consequence of Proposition \ref{lem alt}.
\begin{cor} \label{cor big G implies positive entropy}
For all  $C' \gg 1$ the following is true. If we have at least one of the following :

(1) there exists $0 \leq i \leq K$ such that  $|\cal G_i | \geq  \frac{\eta^{K-i}\varepsilon}{(K+1)l_i}$,

(2) there exists $0 \leq i \leq K-1$ such that $m( B ( \cal G_i, D^{-Tower(Q_0, Q_1,K+3)} )) > \varepsilon $,

then $f$ has a heteroclinic intersection, in which case $f$ has positive topological entropy.

\end{cor}

We include the proof of Corollary \ref{cor big G implies positive entropy} in Appendix \ref{AppB}.

\begin{rema} \label{remarc'}
Given $A,D,h$ as in Theorem \ref{mainA} or \ref{mainB}, we will choose $C'$ to be sufficiently large so that the conclusions of both Lemma \ref{lemmacollectionsofproperties} and Corollary \ref{cor big G implies positive entropy} hold.
\end{rema}

\section{An iterative decomposition}

Now we say a few words about the general strategy behind the proof of Theorem \ref{mainA} and Theorem \ref{mainB}. 
We will inductively define a sequence of decompositions of the surface $X$, denoted by $X = M_{i} \sqcup E_i$.
To start the induction, we define $M_0 = X$ and $E_0 = \emptyset$. Assume that for $k \geq 0$, we have defined $M_k, E_k$ satisfying the following condition: 
\smallskip

\centerline{\begin{textit}
{For each $x \in M_k$, we have $\norm{Df^{q_k}(x)} \leq A^{q_k \theta_0^{k}}$. }
\end{textit} }
\smallskip

Then $E_{k+1}$ is defined as the set of the points that up till some finite time scale, either run into $E_{k}$ with frequency $ \geq \eta$, or is shadowed by hyperbolic orbits ( of course the first case does not happen if $E_k$ is empty ). We will use Proposition \ref{lem hyp pt} to show that the complement of $E_{k+1}$, defined as $M_{k+1}$, again satisfies the induction hypothesis. We then argue that after roughly $K = O(\log (\frac{\log A}{h}))$ steps, $E_{K+1}$ has to be large.  This will show that at some previous time scale, there are enough different hyperbolic hyperbolic points to create a homoclinic intersection.

The formal construction is the following.
 For all $0 \leq k \leq K+1$, we define $M_{k}, E_{k}$ through the following inductive formula. Let
\begin{eqnarray}
E_0 = \emptyset, \quad M_0 = X  \label{def init}
\end{eqnarray}
and for all $0 \leq k \leq K$, we define
\begin{eqnarray} \label{defofek+1}
E_{k+1} &=& \mathcal{R}(E_{k}, \eta, l_k, f^{q_k})  \label{def E k+1} \bigcup  \cal{F}_k,\\
M_{k+1} &=& X \setminus E_{k+1} \label{def M k+1}
\end{eqnarray}


\begin{lemma} \label{lem small der}
If $n > Tower(Q_0, Q_1, K+3)$, then for any $0 \leq k \leq K+1$ we have 
  \begin{align*} x \in M_{k}&\implies \norm{Df^{q_{k}}(x)} \leq A^{q_{k}\theta_0^{k}}.
  \end{align*}
\end{lemma}
\begin{proof}
This is clear when $k=0$ by $\norm{Df} \leq A$ and sub-multiplicativity. Assume that the lemma is valid for some integer $k \in \{0\cdots, K\}$, then $\{ x | \norm{Df^{q_{k}}(x)} > A^{q_k \theta_0^{k}}\} \subset E_{k}$ ( we consider the inclusion valid if both sides are empty). By Corollary \ref{corofproplemhyppt} and \eqref{def E k+1}, we see that any $x \in X$ such that $\norm{Df^{q_{k+1}}(x)} > A^{q_{k+1}\theta_0^{k+1}}$ is contained in $E_{k+1}$. This completes the induction, thus finishes the proof.\end{proof}

We will give the proof of Theorem \ref{mainA} and \ref{mainB} in the next two subsections.
 In the following, we let $C, \theta_0$ be defined in Proposition \ref{lem hyp pt}, let $A,D,h > 0$ be given by Theorem \ref{mainA} or \ref{mainB}, and let $C'$ be sufficiently large depending only on $X$, satisfying Remark \ref{remarc'}.
 
\subsection{Proof of Theorem \ref{mainA}}




 
\begin{prop} \label{propK} Let $C_0$ in Theorem \ref{mainA} be sufficiently large. Then under the conditions of Theorem \ref{mainA},  we have
\begin{eqnarray*}
|E_{K+1}| \geq \varepsilon.
\end{eqnarray*}
\end{prop}

\begin{proof}
We first show the following lemma.

\begin{lemma} \label{lem big ball}
Let $C_0$ in Theorem \ref{mainA} be sufficiently large, and let $n$ be given as in Theorem \ref{mainA}. Then for each $y \in M_{K+1}$, we have 
\begin{eqnarray*}
B(y, e^{-2nh/5} \delta ) \subset B_f(y, n, \delta).
\end{eqnarray*}
\end{lemma}

\begin{proof}

It is clear from \eqref{def l q} that 
\begin{eqnarray*}
l_{K} \in (\frac{n}{q_{K}}, \frac{21}{20} \frac{n}{q_{K}} ).
\end{eqnarray*}

Let $y \in M_{K+1}$. For each $0 \leq i \leq l_K-1$, we denote by
\begin{eqnarray*}
a_i = \log \norm{Df^{q_{K}}(f^{iq_{K}}(y))} , \quad
\delta_{i} = e^{-2nh/5 + iq_{K}h/24+\sum_{j=0}^{i-1} a_i} \delta, \quad
B_i = B(f^{iq_{K}}(y), \delta_{i}).
\end{eqnarray*}
By letting $C_0$ in Theorem \ref{mainA} be sufficiently large, we can ensure that $n > Tower(P_0, P_1, K_0) > Tower(Q_0, Q_1, K+3)$.
Then by Lemma \ref{lemmacollectionsofproperties}(1) and Lemma \ref{lem small der}, we have for each $z \in M_{K}$, $ \log \norm{Df^{q_{K}}(z)} \leq q_{K} \theta_0^{K}\log A \leq \frac{hq_{K}}{4}$.
Then by $y \in M_{K+1}$, \eqref{defofek+1} and Lemma \ref{lem small der}, we have  $
y \notin \mathcal{R}(\{z | \log \norm{Df^{q_K}(z)} > \frac{hq_K}{4}\}, \eta, l_k, f^{q_k})$,
thus
\begin{eqnarray*}
|\{ 0 \leq i \leq l_K-1 |  a_i > \frac{hq_{K}}{4} \}| \leq \eta l_K.
\end{eqnarray*}
Since $0 \leq a_i \leq q_{K}\log A$ for any $0 \leq i \leq l_K-1$, we have
\begin{eqnarray*}
\sum_{j=0}^{i-1} a_j \leq \sum_{j=0}^{l_K-1} a_j \leq \eta l_K q_K \log A +  \frac{l_Kq_{K}h}{4} \leq  \frac{7l_Kq_{K}h}{24}, \quad \forall 0 \leq i \leq l_K-1.
\end{eqnarray*}
The last inequality follows from $\eta \leq \frac{h}{24 \log A}$ which is a consequence of \eqref{def K Delta}, \eqref{def eta L} and $h \in (0, \log A]$.
Then for any $0 \leq i \leq l_K-1$, we have
\begin{eqnarray} \label{term 100}
\delta_i \leq e^{-2nh/5 + l_Kq_{K}h/3} \delta \leq e^{-\frac{1}{20}nh}\delta.
\end{eqnarray}
We claim that for any integer $0 \leq i \leq l_K-1$, 
\begin{eqnarray} \label{lab pass}
f^{iq_{K}}(B_0) \subset B_i.
\end{eqnarray}

We first show that the above claim  concludes the proof of our lemma. Indeed, for any $0 \leq l \leq n$, there exist $0 \leq i \leq l_K-1, 0 \leq j \leq q_{K}-1$ such that $l = iq_{K} + j$. Then we have
\begin{eqnarray*}
f^{l}(B_0)= f^{j}(f^{iq_{K}}(B_0)) \subset f^{j}(B_{i}) \subset B(f^{l}(y), \delta)
\end{eqnarray*}
The last inclusion follows from $A^{q_{K}}\delta_i \leq A^{q_{K}}e^{-nh/20} \delta \leq \delta$, by $\norm{Df^j} \leq A^{q_{K}}$,  \eqref{term 100} and $\frac{n}{q_{K}} \geq \frac{20\log D}{h}$.

Now we obviously have \eqref{lab pass} for $i=0$. Assume that we have \eqref{lab pass} for some $0 \leq i \leq l_K-1$, we will show that we have \eqref{lab pass} for $i+1$. It suffices to show that $f^{q_{K}}(B_i) \subset B_{i+1}$. Using the $C^{2}$ bound $\norm{D^2f^{q_{K}}} \leq D^{q_{K}}$ and $\frac{n}{q_{K}} \geq \frac{20\log D}{h}$, we see that for any $z \in B_i$, 
\begin{eqnarray*}
\norm{Df^{q_{K}}(z)} &\leq& e^{a_i} + \delta_i D^{q_{K}}\\
&\leq& e^{a_i} + D^{q_{K}} e^{-nh/20} \delta 
\leq e^{a_i + hq_{K}/24}.
\end{eqnarray*}
Since $\delta_{i+1} = e^{a_i + hq_{K}/24} \delta_{i}$, we obtain $f^{q_{K+1}}(B_i) \subset B_{i+1}$. This proves \eqref{lab pass} and concludes the proof of Lemma \ref{lem big ball}. \end{proof}

To proceed with the proof of Proposition \ref{propK}, observe that by Lemma \ref{lem big ball},  $M_{K+1} = X\setminus E_{K+1}$ can be covered by $ce^{4nh/5}\delta^{-2}$ many Bowen's $(n,\delta)-$balls. By \eqref{defQ}, $n > P_0$ and by letting $C_0$ be large, we have $c\delta^{-2} < e^{P_0h/5} < e^{nh/5}$.
This implies that $|E_{K+1}| \geq \varepsilon$.  
\end{proof}

\begin{proof}[Proof of Theorem \ref{mainA}] 
Since $f$ is area preserving, by Markov's inequality we have 
\begin{eqnarray*}
|\mathcal{R}(E_{k}, \eta, l_k, f^{q_k})| \leq \eta^{-1}|E_{k}|.
\end{eqnarray*}
Again by the fact that $f$ is area preserving, we obtain the following inequality by \eqref{def E k+1},  \eqref{def F k}
\begin{eqnarray} \label{ineq E} 
|E_{k+1}| \leq\eta^{-1}|E_{k}|  + |\cal{F}_k| 
 \leq \eta^{-1}|E_{k}|  + l_k | \cal{G}_k|.
\end{eqnarray}

By \eqref{ineq E} and \eqref{def init}, we have
\begin{eqnarray*}
|E_{K+1}| \leq \sum_{i=0}^{K} \eta^{i-K} l_i |\cal G_i|.
\end{eqnarray*}
Thus $|E_{K+1}| \geq \varepsilon$ implies that $|\cal G_i| \geq \eta^{K-i} \frac{\varepsilon}{(K+1)l_{i}}$ for some $0 \leq i \leq K$, which by Corollary \ref{cor big G implies positive entropy} (1)  implies that $f$ has positive entropy.
\end{proof}

\subsection{Proof of Theorem \ref{mainB}}

The proof of Theorem \ref{mainB} is parallel to that of Theorem \ref{mainA}.
The following proposition is an analogue of Proposition \ref{propK}.
\begin{prop} \label{lem tilde e big}
Let $C_0$ in Theorem \ref{mainB} be sufficiently large, and let $n$ be as in  Theorem \ref{mainB}. Then under the condition of Theorem \ref{mainB}, we have
\begin{eqnarray*}
\mu_{x,n}(E_K) \geq \frac{h}{2 \log A }.
\end{eqnarray*}
\end{prop}

\begin{proof}
By letting $C_0$ in Theorem \ref{mainA} be sufficiently large, we can ensure that $n > Tower(P_0, P_1, K_0) > Tower(Q_0, Q_1, K+3)$.
Then by Lemma \ref{lem small der}, for each $y \in M_{K}$, we have $\norm{Df^{q_K}(y)} \leq A^{q_K \theta_0^{K}} \leq e^{\frac{q_Kh}{2}}$. 

We take a subset $\{p_1, \cdots, p_{l}\} \subset \{0,\cdots, n-q_K\}$ so that $I_j :=\{p_j,\cdots, p_j + {q_K}-1\}, 1 \leq j \leq l $ are disjoint subsets of $\{0,\cdots, n-1\}$ and $f^{p_j}(x) \in M_{K}$ for all $j$. Moreover, we assume that  for any $k \in \{0,\cdots, n-1\} \setminus \bigcup_{j=1}^l I_j$, we have $f^{k}(x) \in E_K$. The construction of $\{p_i\}_{i=1}^{l}$ is straightforward.

Then by sub-multiplicativity, we have
\begin{eqnarray*}
\log\norm{Df^{n}(x)} &\leq& \sum_{i=1}^{l} \log \norm{Df^{q_K}(f^{p_i}(x))} + (n - l q_K) \log A \\
&\leq& \frac{1}{2} l  h q_K + (n-lq_K) \log A.
\end{eqnarray*}
By the condition in Theorem \ref{mainB}, we have $\log \norm{Df^{n}(x)} > n h$.
Thus $n(\log A - h) > l q_K( \log A - \frac{h}{2})$.
Then we see that $\mu_{x,n}(E_K) \geq \frac{n-lq_K}{n} \geq \frac{h}{2 \log A} $. 
\end{proof}

\begin{proof}[Proof of Theorem \ref{mainB}]
For any measurable set $B \subset X$, any integers $n,l \geq 1$, any $x \in X$, we have
\begin{eqnarray*}
\mu_{x,n}(f^{-l}(B)) \leq \frac{l}{n} + \mu_{x,n}(B).
\end{eqnarray*}
Then for any $k=0,\cdots, K-1$, by Markov's inequality we have
\begin{eqnarray*}
\mu_{x,n}(\cal R(E_k, \eta, l_k, f^{q_k})) &\leq&(\eta l_k) ^{-1} \int \sum_{i=0}^{l_K-1} 1_{f^{-iq_{K}}(E_k)}) d\mu_{x, n}\\
&\leq& (\eta l_k) ^{-1} \sum_{i=0}^{l_k-1} \mu_{x,n}(f^{-iq_{k}}(E_k)) \\
&\leq& (\eta l_k) ^{-1} \sum_{i=0}^{l_k-1} (\mu_{x,n}(E_k) + \frac{iq_{k}}{n}) \leq \eta^{-1}\mu_{x,n}(E_k) + \frac{q_{k+1}}{2n \eta}.
\end{eqnarray*}
Similarly, we have
\begin{eqnarray*}
\mu_{x,n}(\cal {F}_k) &\leq& \sum_{i=0}^{l_k-1} \mu_{x,n}(g^{-iq_k}(\cal G_k)) \\ 
&\leq& l_k \mu_{x,n}(\cal G_k)  + \frac{l_kq_{k+1}}{n}.
\end{eqnarray*}
Then we have an  inequality analogous to \eqref{ineq E}, as follows,
\begin{eqnarray*} \label{ineq E2}
\mu_{x,n}(E_{k+1}) &\leq&   \mu_{x,n}(\cal R(E_k, \eta, l_k, f^{q_k}))+ \mu_{x,n}(\cal {F}_k)  \\
&\leq& \eta^{-1} \mu_{x,n}(E_k) + \frac{q_{k+1}}{2n\eta} + l_k \mu_{x,n}(\cal G_k)  + \frac{l_kq_{k+1}}{n}. \nonumber
\end{eqnarray*}
By the simple observation that $l_k \geq l_0 \geq \eta^{-1}$ for all $0 \leq k \leq K$, we have
\begin{eqnarray*}
\mu_{x,n}(E_K) \leq \sum_{i=0}^{K-1} \eta^{-K+i+1}( l_i \mu_{x,n}(\cal G_i) + \frac{2l_{i}q_{i+1}}{n}).
\end{eqnarray*}
By \eqref{def l q} and Proposition \ref{lem tilde e big}, we see that there exists $0 \leq i \leq K-1$ such that
\begin{eqnarray*}
\mu_{x,n}(\cal G_i) &\geq& l_i^{-1}(\frac{\eta^{K-i-1}}{K}\frac{h}{2\log A } - \frac{2q_{K}l_{K-1}}{n})\geq l_i^{-2}.
\end{eqnarray*}
The last inequality follows from
\begin{eqnarray*}
\frac{2q_K l_{K-1}}{n} < \frac{2l_{K-1}}{l_K} < q_0^{-1} \leq e^{-C' (\log \Delta)^2} < \frac{\eta^{K}h}{4K \log A}, \\
\frac{\eta^{K}h}{4K \log A} > l_0^{-1} \geq l_i^{-1}, \quad \forall 0 \leq i \leq K-1,
\end{eqnarray*}
by letting $C'$ be larger than some absolute constant.
In particular, by Lemma \ref{lemmacollectionsofproperties}(2), \eqref{def K Delta}, \eqref{initial q}, \eqref{def l q}, and by letting $C_0$ in Theorem \ref{mainB} be sufficiently large, we have  
\begin{eqnarray*} 
&\mu_{x,n}(\cal G_i) >  Tower(Q_0, Q_1, K+1)^{-2} >  Tower(Q_0, Q_1, K+2)^{-1} , \\
&K_0 \geq K +4, \quad P_i > Q_i, i=0,1.
\end{eqnarray*}
By the condition of Theorem \ref{mainB} that $x$ is $(n, Tower(P_0,P_1,K_0-1)^{-1}, D^{-Tower(P_0,P_1,K_0-1)},\varepsilon)$-sparse, we see that 
\begin{eqnarray*}
m(B(\cal G_i, D^{-Tower(Q_0, Q_1,K+3)})) \geq m(B(\cal G_i, D^{-Tower(P_0, P_1,K_0-1)})) > \varepsilon.
\end{eqnarray*}
This concludes the proof by Corollary \ref{cor big G implies positive entropy} (2).

\end{proof}

\appendix

\section{}

In this section we prove the main technical result Proposition \ref{lem hyp pt}.
We start with a slight generalization of Pliss lemma \cite{pliss}.
\begin{lemma}[a variant of Pliss] \label{lem improved pliss}
For any real numbers $N \geq 1$, $1 > \theta_0 > \theta_1 > \theta_2 > 0$, $\eta \in (0,\frac{1}{2} \frac{1-\theta_0}{N-\theta_0}\frac{ \theta_1 -   \theta_2}{N - \theta_2})$, for any integer $n \geq 1$, real number $l > 0$, the following is true.
Given a sequence of $n$ real numbers $a_1,...,a_n$. Assume that 
\begin{enumerate}
\item $a_i \leq Nl$ for all $1\leq i \leq n$,
\item $\sum_{i=1}^{n}a_i >n \theta_1l$,
\item $|\{ 1 \leq i \leq n | a_i > \theta_0 l\}|<  \eta n$.
\end{enumerate}
Then there exist at least $\frac{\theta_1-\theta_2}{1-\theta_2}n$ indexes $i$'s such that $\frac{1}{k}\sum_{j=i}^{i+k-1} a_j > \theta_2 l$ for all $1 \leq k\leq n+1-i$.
\end{lemma}
\begin{proof}
Denote by $$A := \{i | \mbox{ there exists $1 \leq k \leq n+1-i$ such that $\frac{1}{k}\sum_{j=i}^{i+k-1} a_j \leq \theta_2l$}\} $$
Without loss of generality, we assume that $A \neq \emptyset$, for otherwise the conclusion of Lemma \ref{lem improved pliss} is already true.
Then $A$ is contained in a non-empty set $I \subset \{1,\cdots, n\}$ satisfying that $\frac{1}{|I|} \sum_{i \in I} a_i \leq \theta_2 l$. Then by (1),(2), we obtain that 
\begin{eqnarray*}
l N |I^{c}| + l \theta_2|I| >  l n \theta_1.
\end{eqnarray*}
By $l > 0$, the above inequality implies that $|I^c| \geq \frac{ \theta_1 -   \theta_2}{N- \theta_2}n$.
We claim that 
\begin{eqnarray} \label{lab improved pliss}
\frac{1}{|I^{c}|}\sum_{i \in I^{c}} a_i \leq  l.
\end{eqnarray}
Indeed, if \eqref{lab improved pliss} was false, by (1) we would have at least $\frac{1-\theta_0}{N-\theta_0}|I^{c}| \geq \frac{1-\theta_0}{N-\theta_0}\frac{ \theta_1 -   \theta_2}{N- \theta_2}n > \eta n$ indexes $i \in I^{c}$ such that $a_i > \theta_0 l$, but this would contradict (3).

Now we use (2) again, with the improved estimate \eqref{lab improved pliss} in place of (1), and obtain
\begin{eqnarray*}
l |I^{c}| + \theta_2 l |I| \geq \sum_{i \in I^{c}} a_i + \sum_{i \in I} a_i > n \theta_1 l.
\end{eqnarray*}
This implies that $|I^c| \geq \frac{\theta_1 - \theta_2}{ 1 - \theta_2}n$. We conclude the proof by the definition of $I$.
\end{proof}

Let $x$ be given by the condition of Proposition \ref{lem hyp pt}.
We will define a collection of charts along a sub-orbit of $x$ following the definitions and estimates in \cite{AFLXZ}. 

Let $v_{s}$ be a unit vector in the most contracting direction of $Dg^{q}(x)$ in $T_{x}X$, and let $v_{u}$ be a unit vector orthogonal to $v_{s}$. For each $0 \leq i \leq q$, we define
\begin{eqnarray*}
v^s_i&:=&\frac{Dg^{i}(v_s)}{\norm{Dg^{i}(v_s)}}, \quad v^u_i:=\frac{Dg^{i}(v_u)}{\norm{Dg^{i}(v_u)}}, \\ 
\lambda^s_{i}&:=&\log\frac{\norm{Dg(v^s_i)}}{\norm{v^s_i}}, \quad \lambda^u_{i} :=\log\frac{\norm{Dg(v^u_i)}}{\norm{v^u_i}}, \\
\overline{\lambda}^{e}_i &:=& \min\{ \lambda^u_i, -\lambda^s_i \}.
\end{eqnarray*}

Given $r > 0, \tau >0, \kappa>0$, we define a {\it $(r, \tau,\kappa)-$Box }, which we denote by $U(r,\tau,\kappa)$, to be
\begin{eqnarray*}
U(r,\tau,\kappa) = \{(v,w) \in \R^2 | \norm{v} \leq r, \norm{w} \leq \tau + \kappa \norm{v} \}.
\end{eqnarray*}

For $\kappa>0$, we denote by
\begin{eqnarray*}C(\kappa)&=&\{(v,w)\in \mathbb{R}^2 | \norm{w} < \kappa \norm{v} \},\\
\tilde{C}(\kappa)&=&\{(v,w)\in \mathbb{R}^2 | \norm{v} < \kappa \norm{w} \}.
\end{eqnarray*}
We will refer to these sets as cones.

We now recall some definitions in \cite{AFLXZ}.

\noindent $\bullet$ A curve contained in $\mathbb{R}^2=\mathbb{R}_x\oplus \mathbb{R}_y$ is called a {\it $\kappa-$horizontal graph} if it is the graph of a Lipschitz function from an closed interval $I\subset\mathbb{R}_x$ to $\mathbb{R}_y$ with Lipschitz constant less than $\kappa$. Similarly, we can define the {\it $\kappa-$ vertical graphs}.

\medskip 
\noindent $\bullet$ The boundary of an $(r,\tau,\kappa)-$Box $U$  is the union of two $0-$ vertical graphs and two $\kappa-$ horizontal graphs. We call these graphs respectively, the {\it left (resp. right) vertical boundary of $U$ } and the {\it upper (resp. lower) horizontal boundary of $U$}. We call the union of the left and right vertical boundary of $U$ the {\it vertical boundary of $U$}. Similarly we call the union of the upper and lower horizontal boundary of $U$ the {\it horizontal boundary of $U$}.

\medskip 
\noindent $\bullet$ Horizontal and vertical graphs which connect the boundaries of $U$ will be called full horizontal and  full vertical graphs as in the following definition. Given $r,\tau,\kappa, \eta >0$,
for each $(r,\tau, \kappa)-$Box $U$, an {\it $\eta-$ full horizontal graph of $U$} is an $\eta-$  horizontal graph $L$ such that $L \subset U$ and the endpoints of $L$ are contained in the vertical boundary of $U$. Similarly, we define  {the \it $\eta-$full vertical graphs of $U$}.

\medskip 
\noindent $\bullet$  We define an {\it $\eta-$horizontal strip of $U$} to be a subset of $U$ bounded by the vertical boundary of $U$ and two disjoint $\eta-$ full horizontal graphs of $U$ which are both disjoint from the horizontal boundary of $U$.
Similarly we can define {\it $\eta-$vertical strips of $U$}. Like Boxes, we define the horizontal, vertical boundary of a strip.

\medskip 
\noindent $\bullet$ Given a Box $U$, $\cal{R}'$ a vertical strip of $U$, and $\cal{R}$ a horizontal strip of $U$, a homeomorphism that maps  $\mathcal{R}'$ to $\mathcal{R}$ is said to be {\it regular} if it maps  the horizontal (resp. vertical) boundary of $\mathcal{R}'$ homeomorphically to the horizontal (resp. vertical) boundary of $\mathcal{R}$.

We recall the definition of hyperbolic map in \cite{AFLXZ}.
\begin{defi} \label{def hyp map}
Given $r,\tau >0, 0 <\kappa,\kappa',\kappa''<1$. Denote $U =U(r,\tau,\kappa)$, and let $\mathcal{R}_1$ be a $\kappa-$vertical strip of $U$, $\mathcal{R}_2$ be a $\kappa-$horizontal strip  of $U$. A $C^1$ diffeomorphism $G : \mathcal{R}_1 \to \R^2$ is called a hyperbolic map if it satisfies the following conditions:
\begin{eqnarray}
\label{R1R2} \mbox{$G$ is a regular map from $\mathcal{R}_1$ to $\mathcal{R}_2$},\\
\label{preserve horizontal cone}\forall x\in \mathcal{R}_1, DG_x(C(\kappa'))\subset C(\frac{1}{2}\kappa'),\\ 
\label{preserve vertical cone}\forall x\in \mathcal{R}_2, DG^{-1}_x(\tilde{C}(\kappa")) \subset \tilde{C}(\frac{1}{2}\kappa").
\end{eqnarray}
The following is a sketch of a hyperbolic map.
\begin{figure}[ht!]
\centering
\includegraphics[width=120mm]{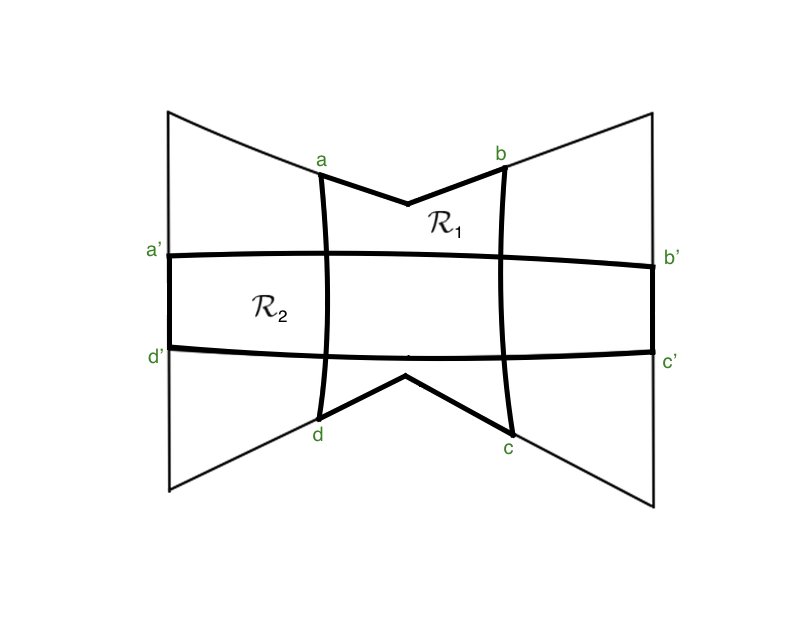}
\caption{ $\mathcal{R}_1$ is the topological rectangle $abcd$,  $\mathcal{R}_2$ is the topological rectangle $a'b'c'd'$. Under a hyperbolic map $G$, $ab$ is mapped to $a'b'$. Similarly, $bc,cd,da$ are mapped respectively to $b'c',c'd',d'a'$. \label{overflow}}
\end{figure}

\end{defi}

For each $0 \leq n \leq q$, we define $i_n : \R^2 \to T_{x_n}S$ as 
\begin{eqnarray*}
i_n(a,b) = a v^{u}_n  + bv^{s}_n.
\end{eqnarray*} 
There exists a constant $R = R(X) > 0$ such that :
$exp_{x_n}$ is a diffeomorphism restricted to $i_n(B(0, D_1^{-\Delta}R))$ and $exp_{x_{n+1}}^{-1}$ is a diffeomorphism restricted to $gexp_{x_n}i_n(B(0,D_1^{-\Delta}R))$. Denote by $g_n$ the $C^2$ diffeomorphism 
\begin{eqnarray*}
\label{def g n}  g_n: B(0,D_1^{-\Delta}R) &\to& \mathbb{R}^2 \\
g_n(v,w)&=&i_{n+1}^{-1}\exp_{x_{n+1}}^{-1}g\exp_{x_n} i_n(v,w). \nonumber
\end{eqnarray*}

We set $M:=1000$, and
\begin{eqnarray*}
\bar{r}  = D_1^{-3\Delta M}, \quad \bar{\kappa}  = D_1^{-\Delta M}, \quad \delta = \frac{\log A_1}{100}.
\end{eqnarray*}

The main estimates in \cite{AFLXZ} are summarised in the following proposition.
\begin{prop} \label{main prop in aflxz}
Under the conditions of Proposition \ref{lem hyp pt} for some absolute constant $\theta_0 \in (0,1)$ sufficiently close to $1$, and $C > 0$ sufficiently large depending only on $X$, 
there exist constant $C_1 = C_1(X)$, integers $0 \leq i_1 \leq i_2 \leq q$, and sequences of positive numbers $\{(r_n,\tau_n, \kappa_n, \tilde{\kappa)}_{n} \}_{i_1 \leq n \leq i_2}$ such that :
\begin{enumerate}
\item (Positive proportion)
\begin{eqnarray*}
i_2 - i_1 \geq D_1^{-C_1\Delta}q,
\end{eqnarray*}
\item (Tameness at the starting and ending points )
\begin{eqnarray*}
&&\cot\angle(E^u_{i_1}, E^s_{i_1}), \cot\angle(E^u_{i_2}, E^s_{i_2}) < \frac{D_1^{M\Delta }}{100}, \\
&&10^{6}\bar{r} \geq r_i \geq \tau_i, \quad \forall i_1 \leq i \leq i_2\\
&&r_{i_1} =  \tau_{i_1} = \bar{r}, \quad \kappa_{i_1} = \tilde{\kappa}_{i_1} =\bar{\kappa}, \\
&&r_{i_2} = 10^{6}\bar{r},  \quad \tau_{i_2} \leq \frac{1}{10}\bar{r}, \quad \kappa_{i_2} = \frac{1}{100}\bar{\kappa}, \quad \tilde{\kappa}_{i_2} = 100\bar{\kappa},  \\
&&\sum_{n=i_1}^{i_2-1} \lambda_n^{u}, \sum_{n=i_1}^{i_2-1} -\lambda_n^{s}  \geq \frac{2}{3}(i_2-i_1)a,
\end{eqnarray*}
\item (Transversal mappings) Let $r_n, \tau_n, \kappa_n$ be as above, we let
\begin{eqnarray*}
U_n = U(r_n, \tau_n, \kappa_n), \quad
C_n = C(\kappa_n), \quad
\tilde{C}_n = \tilde{C}(\tilde{\kappa}_n) 
\end{eqnarray*}
 If $\Gamma$ is a $\kappa_n-$full horizontal graph of $U_n$, then $g_n(\Gamma) \bigcap U_{n+1}$ is a $\kappa_{n+1}$-full horizontal graph of $U_{n+1}$. Moreover, the image of the horizontal boundary of $U_n$ under $g_n$ is disjoint from the horizontal boundary of $U_{n+1}$; the image of the vertical boundary of $U_n$ under $g_n$ is disjoint from the vertical boundary of $U_{n+1}$.

\item (Cone condition)
Furthermore, for any $(v,w) \in U_n$, we have $(Dg_n)_{(v,w)}(C_n) \subset C_{n+1}$; for any $(v,w) \in g_n(U_n) \bigcap U_{n+1}$, we have $(Dg_n^{-1})_{(v,w)}(\tilde{C}_{n+1}) \subset \tilde{C}_{n}$. Moreover, for  any $(v,w) \in U_n$, any $(V,W) \in C_{n}$, let $(\bar{V}, \bar{W}) = (Dg_n)_{(v,w)}(V,W)$, we have $|\bar{V}| \geq e^{\lambda_n^{u}- \delta}|V|$;  for  any $(v,w) \in g_n(U_n) \bigcap U_{n+1}$, any $(V,W) \in \tilde{C}_{n+1}$, let $(\bar{V}, \bar{W}) = (Dg_n^{-1})_{(v,w)}(V,W)$, we have $|\bar{W}| \geq e^{-\lambda_n^{s} - \delta}|W|$.

\item (Hyperbolic map)
 Denote  \begin{eqnarray*} 
 J &=& i_{i_1}^{-1} exp_{x_{i_1}}^{-1} exp_{x_{i_2}} i_{i_2}, \\
 G &=& { i_{i_1}^{-1}exp_{x_{i_1}}^{-1}} g^{i_2-i_1}exp_{x_{i_1}}i_{i_1} = J g_{i_2-1}\cdots g_{i_1}.
 \end{eqnarray*} 
There exist $\cR_1$, a $100\bar{\kappa}-$vertical strip of $U_{i_1}$, and  $\cR_2$, a $100\bar{\kappa}-$horizontal strip of $U_{i_1}$ such that $G$ is a hyperbolic map from $\mathcal{R}_1$ to $\mathcal{R}_2$ with parameters $\kappa'= \bar{\kappa} , \kappa'' = 100 \bar{\kappa}$. Moreover, for each $0 \leq j \leq i_2-i_1$, we have $g_{i_1+j-1}\cdots g_{i_1}(\cR_1) \subset U_{i_1+j}$.
\end{enumerate}

\end{prop}

We will give a sketch of the proof and refer the detailed estimates to \cite{AFLXZ}.
\begin{proof}
Set  $a = \log A_1$. Condition (4) in Proposition \ref{lem hyp pt} translates into 
\begin{eqnarray*}
\frac{1}{q} \sum_{i=0}^{q-1} \lambda_i^{s} \leq -a,  \quad \frac{1}{q} \sum_{i=0}^{q-1} \lambda_i^{u} \geq a.
\end{eqnarray*}

Using condition (3) and Lemma \ref{lem improved pliss} in place of the Pliss lemma, by setting $\theta_0 \in (0,1)$ to be an absolute constant sufficiently close to $1$, and setting $C > 0$ to be sufficiently large depending only on $X$, we can show analogously to Lemma 4.4 in \cite{AFLXZ} , that there are more than $q/2$ points in $\{g^{k}(x)| 0 \leq k \leq q-1\}$ that are \lq\lq good in the orbit of $x$\rq\rq. Here a point $g^{n}(x)$ is said to be \textit{good in the orbit of $x$} if $n \in [1, q-1]$ satisfies the following conditions :
\begin{eqnarray}
\frac{1}{k}\sum_{j=n}^{n+k-1}\overline{\lambda}_{j}^e&>&(1-\frac{1}{1000})\theta_0^{-1}a, \forall 1\leq k\leq q-n,\label{fwd good 2}\\
\frac{1}{k}\sum_{j=n-k}^{n-1}\overline{\lambda}_{j}^e&>&(1-\frac{1}{1000})\theta_0^{-1}a, \forall 1\leq k\leq n.\label{bwd good}
\end{eqnarray}
We can show in analogy to Lemma 4.5 that $|\cot \angle(v^s_n, v^u_n)| \leq A_1^{3\Delta }$ for all $n$ such that $g^{n}(x)$ is good in the orbit. Then there exist an integer $0 \leq i \leq D_1^{-C_1 \Delta}q $ such that the subsequence $(x_{i + j \lceil D_1^{-C_1 \Delta}q\rceil} )_{0 \leq j \leq \lfloor \frac{q}{\lceil D_1^{-C_1 \Delta}q \rceil} \rfloor-1} $ contains at least $\frac{1}{3}D_1^{C_1 \Delta}$ many points which are good in the orbit of $x$.  By letting $C_1$ to be sufficiently large depending only on $X$, we can apply the pigeonhole principle to the above subsequence as in the proof of Proposition 4.1 in \cite{AFLXZ} and obtain $0 \leq i_1 < i_2 \leq q-1$ that satisfy the following conditions: 
\begin{enumerate}
\item $i_2 - i_1 \geq D_1^{-\Delta C_1}q$,
\item $\sum_{j=i_1}^{i_1+ k-1}\overline{\lambda}_{j}^e>(1-\frac{1}{1000})\theta_0^{-1}ak, \forall 1\leq k \leq i_2-i_1,$
\item $\sum_{j=i_2-k}^{i_2-1}\overline{\lambda}_{j}^e>(1-\frac{1}{1000})\theta_0^{-1}ak, \forall 1\leq k\leq i_2-i_1,$
\item The angles $\angle(v^s_{i_1}, v^u_{i_1}), \angle(v^s_{i_2},v^u_{i_2})$ satisfy
\begin{eqnarray*}
\log|\cot\angle(v^s_{i_1}, v^u_{i_2})|\leq 3\Delta a, \quad \log|\cot\angle(v^s_{i_2}, v^u_{i_2})|\leq 3\Delta a,
\end{eqnarray*}
\item Moreover, we have $d(g^{i_1}(x), g^{i_2}(x)) < D_1^{-\frac{C_1 \Delta}{200}}$, and 
\begin{eqnarray*}
d_{T^1X}(v^s_{i_1}, v^s_{i_2})<D_1^{-\frac{C_1 \Delta}{200}}, \quad d_{T^1X}(v^u_{i_1}, v^u_{i_2})<D_1^{-\frac{C_1 \Delta}{200}}.
\end{eqnarray*}
\end{enumerate}
We note the similarities between the above conditions and those of Definition 4.3 in \cite{AFLXZ}. However here we have a large inverse power of $D_1$ in (5) instead of a small inverse power of $q$ as in Definition 4.3, (4) in \cite{AFLXZ}. This is sufficient for the rest of proof, since $r_{i_1}$,$r_{i_2}$, $\angle( E^s_{i_1}, E^u_{i_1})$ and $\angle( E^s_{i_2}, E^u_{i_2})$ are lower bounded by  $D^{-O(\Delta )}$.

At this point, we can invoke the proof of Proposition 4.2, and obtain (2) as a consequence of Lemma 4.6, 4.7, 4.8 in \cite{AFLXZ}; and obtain (3),(4) as a consequence of Proposition 4.5 in \cite{AFLXZ}. We obtain (5) following the proof of Proposition 4.4 in \cite{AFLXZ}. 
\end{proof}

Now we are ready to conclude the proof of Proposition \ref{lem hyp pt}.
\begin{proof}[Proof of Proposition \ref{lem hyp pt}]
We apply Proposition \ref{main prop in aflxz} and obtain $i_1$,$i_2$, $\cal{R}_1$, $\cal{R}_2$, $G$, $U_i$, $C_i$, $\tilde{C}_i$ as in the proposition. We set $i= i_1, j= i_2$. By (5) in Proposition \ref{main prop in aflxz} and Proposition 4.3 in \cite{AFLXZ}, we obtain a hyperbolic periodic point in $\cR_1 \bigcap \cR_2$, denoted by $y$. 

We note the following lemma whose proof follows from the standard construction of unstable / stable manifolds for uniformly hyperbolic maps using graph transform argument.
For this reason, we omit its proof.

\begin{lemma} \label{lem standard graph transform}
Let $r, \tau > 0, L > 1$, $0 < \kappa, \kappa', \kappa'' < 1$, $U = U(r, \tau, \kappa)$ and let  $G : \cal{R}_1 \to \cal{R}_2$ be a hyperbolic map where $\cal R_1$ ( resp. $\cal R_2$) is the $\kappa-$vertical strip (resp. $\kappa-$ horizontal strip ) of $U$ as in Definition \ref{def hyp map}, and $\kappa', \kappa''$ satisfy inclusion \eqref{preserve horizontal cone}, \eqref{preserve vertical cone} respectively. Assume that

(1) For each $x \in \cal R_1$, each $(V, W) \in C(\kappa')$, set $(\bar{V}, \bar{W}) = DG_x(V,W)$, then $|\bar{V}| \geq L|V|$,

(2) For each $x \in \cal R_2$, each $(V, W) \in \tilde{C}(\kappa'')$, set $(\bar{V}, \bar{W}) = DG_x^{-1}(V,W)$, then $|\bar{W}| \geq L|W|$.

 Then there exists a hyperbolic fixed point of $G$,  $y\in \cal R_1 \bigcap \cal R_2$, whose local unstable manifold in $\cal R_2$, denoted by $\cWu_{G}(y)$, is a $\kappa'-$horizontal graph; whose local stable manifold in $\cal R_1$, denoted by $\cWs_{G}(y)$, is a $\kappa''-$vertical graph. Moreover we have $$G(\cal R_1) \subset B(\cWu_G(y), 2L^{-1} diam(U)).$$
\end{lemma}
We set $L = A^{\frac{j-i}{2}}$.
We now verify conditions (1),(2) of Lemma \ref{lem standard graph transform} for $L$, $G$, $U = U_{i_1}$, $\kappa = 100\bar{\kappa}, \kappa' = \bar{\kappa}, \kappa'' = 100\bar{\kappa}$. We only verify condition (2) in details since condition (1) can be verified in a similar fashion. By Proposition \ref{main prop in aflxz}(5), for any $i \leq n \leq j-1$, we have $g_{n+1}^{-1} \cdots g_{j-1}^{-1} J^{-1}(\cal R_2) = g_{n} \cdots g_{i}(\cal R_1) \subset U_{n+1} \bigcap g_{n}(U_{n})$. For any $i \leq n \leq j$, any $(v,w) \in \cal R_2$, for any $(V,W) \in  \tilde{C}_{j}$ ( here $\tilde{C}_j$ is given by Proposition \ref{main prop in aflxz}(3)), denote by $(v_n, w_n) = g_{n}^{-1}\cdots g_{j-1}^{-1}J^{-1}(v,w)$, $(V_n, W_n) = D(J g_{j-1}\cdots g_{n})^{-1}_{(v,w)}(V,W)$. Then we have $(v_n, w_n) \in U_n$ for all $i \leq n \leq j$. By Proposition \ref{main prop in aflxz}(2),(4), we have $|W_i| \geq e^{\sum_{n=i}^{j-1}( -\lambda_n^{s} - \delta)}|W_j| \geq A^{\frac{j-i}{2}}|W| = L|W|$. 

By Lemma \ref{lem standard graph transform} and Proposition \ref{main prop in aflxz}(2), we obtain 
\begin{eqnarray*}
G(\cal{R}_1) \subset B(\cWu_{G}(y), 200A^{-\frac{j-i}{2}}\bar{r})
\end{eqnarray*}
We denote by $z = exp_{x_{i_1}}i_{i_1}(y)$. By Proposition \ref{main prop in aflxz}(5) and the fact that $y$ is a hyperbolic fixed point of $G$, we conclude that $z$ is a $g-$hyperbolic periodic point.
Then by Proposition \ref{main prop in aflxz} and by possibly increasing $C_1$ depending only on $X$, we can ensure that $z \in \cal{H}(g, D_1^{-C_1\Delta})$, and 
\begin{eqnarray*}
g^{j}(x) \in exp_{x_{i_1}}i_{i_1}G(\cal{R}_1) \subset  B(\cWu_{D_1^{-3C_1\Delta}}(z), A_1^{-\frac{ q }{2D_1^{C_1  \Delta}}})
\end{eqnarray*}
We conclude the proof by letting $C$ to be sufficiently large depending only on $X$.

\end{proof}

\section{} \label{AppB}

\begin{proof}[Proof of Corollary \ref{cor big G implies positive entropy}]  
In this following, we briefly denote $H(f, \alpha, r)$ by $H(\alpha, r)$, and denote $H(f, \lambda )$ by $H(\lambda )$.

We first prove the corollary under condition (1).
For any $\alpha, r, \xi > 0$, any $y \in \cal H(\alpha, r)$,
\begin{eqnarray*}
|B(\cWu_r(y), \xi)| \ll r \xi.
\end{eqnarray*}
 It is clear from the definition of $\cal G$ in \eqref{defG} that for any $\lambda \in (0,1)$,
\begin{eqnarray*}
|\cal H(\lambda)| &\geq& |\cal G(\lambda, \xi, f) | /  | B(\cWu_{\lambda^3}(y), \xi) | \\
&\gg&  \lambda^{-3} \xi^{-1} | \cal G(\lambda, \xi, f) |
\end{eqnarray*}
By \eqref{def G k} and Proposition \ref{lem alt}, it suffices to check that $|\cal G_i| \gg A^{-\frac{q_{i+1} \theta_0^{i+1}}{2D^{C \Delta q_i}}}  D^{11C \Delta q_i}$. Since $x^{-1} > e^{-x}$ for $x \in (0, \infty)$, we have 
\begin{eqnarray*}
|\cal G_i| &\geq&  \frac{\eta^{K-i}\varepsilon}{(K+1)l_i} \\ 
&>& (10K)^{-1}\eta^{K} \varepsilon \frac{q_{i+1}\theta_0^{i+1}}{D^{C \Delta q_i}l_i} A^{-\frac{q_{i+1}\theta_0^{i+1}}{4D^{C \Delta q_i}}} \\
&\gg& A^{-\frac{q_{i+1}\theta_0^{i+1}}{2D^{C  \Delta q_i}}}  D^{11C \Delta q_i}.
\end{eqnarray*}
The last inequality follows from by letting $C' \gg 1$, and
\begin{itemize}
\item $K^{-1}\eta^{K}\varepsilon q_{i+1} \theta_0^{i+1} l_{i}^{-1}  \gg 1$  since    by \eqref{initial q}, and $q_{i+1} l_{i}^{-1} \theta_0^{i+1} \geq \frac{1}{2}q_{i} \theta_0^{i} \geq  \frac{1}{2}q_0 \geq  \frac{1}{2}\varepsilon^{-1}e^{C' (\log(\frac{\log A}{h}))^2 + C'} \gg \varepsilon^{-1} K \eta^{-K}$,
\item $A^{\frac{q_{i+1}\theta_0^{i+1}}{4D^{C \Delta q_i}}} \geq D^{12C \Delta q_i}$, by $A \gg 1$, \eqref{HHH} and \eqref{def l q}.
\end{itemize}

Now we consider condition (2). We set $\delta_0 = D^{-Tower(Q_0,  Q_1, K+3)}$.
By Lemma \ref{lemmacollectionsofproperties}(3),  we have 
\begin{eqnarray*}
\delta_0 \leq \xi_0 \mbox{ and }
\delta_0 \leq \lambda^3_i,  \forall 0 \leq i \leq K.
 \end{eqnarray*}
For any $\lambda, \xi \in (0,1)$, any $y \in \cal H( \lambda )$, any $\delta \in (0, \lambda^3 )$, we have
\begin{eqnarray*}
m( B(B(\cWu_{\lambda^3}(y), \xi), \delta )) \ll \lambda^3 \max( \xi, \delta).
\end{eqnarray*}

By \eqref{defG} and condition (2), we have for some $0 \leq i \leq K$ that,
\begin{eqnarray*}
|\cal H( \lambda_i)| &\geq& m( B( \cal G(\lambda_i, \xi_i, f), \delta_0) )  /  \sup_{y \in H(\lambda_i) }m(B(B(\cWu_{\lambda^3_i}(y), \xi_i), \delta_0)) \\
&\gg&  \varepsilon \lambda_i^{-3} \min( \xi_i^{-1} , \delta_0^{-1}).
\end{eqnarray*}

By Proposition \ref{lem alt},  it suffices to observe from  Lemma \ref{lemmacollectionsofproperties}  that
\begin{eqnarray*}
\varepsilon \gg \lambda_i^{-11} \max(\xi_i, \delta_0), \quad \forall 0 \leq i \leq K.
\end{eqnarray*}

\end{proof}

\
\

\

{\footnotesize \noindent Bassam Fayad\\
Institut de Math\'ematiques de Jussieu-Paris Rive Gauche UMR7586 CNRS
Université Paris Diderot-Université Pierre et Marie Curie \\
E-mail: bassam.fayad@imj-prg.fr

\
\

{\footnotesize \noindent Zhiyuan Zhang\\
Institut de Math\'{e}matique de Jussieu---Paris Rive Gauche, B\^{a}timent Sophie Germain, Bureau 652\\
75205 PARIS CEDEX 13, FRANCE\\
Email address: zzzhangzhiyuan@gmail.com

\end{document}